\documentclass{aptpub}

\authornames{I. ARAB, M. HADJIKYRIAKOU, P.E. OLIVEIRA} 
\shorttitle{Failure rate properties of parallel systems} 

\newcommand{\dE}{\mathbb{E}}
\newcommand{\TXs}[1]{\overline{T}_{X,#1}}

\newcommand{\TYs}[1]{\overline{T}_{Y,#1}}
\newcommand{\muXs}[1]{\widetilde{\mu}_{X,#1}}

\newcommand{\sIFR}[1]{\ensuremath{#1\!-\!}{\rm IFR}}
\newcommand{\sDFR}[1]{\ensuremath{#1\!-\!}{\rm DFR}}
\newcommand{\sIFRA}[1]{\ensuremath{#1\!-\!}{\rm IFRA}}
\newcommand{\sDFRA}[1]{\ensuremath{#1\!-\!}{\rm DFRA}}

\newcommand{\dI}{\ensuremath{\mathbb{I}}}

\numberwithin{equation}{section}  

\begin{document}

\title{Failure rate properties of parallel systems} 
\footnote{This work was partially supported by the Centre for Mathematics of the University of Coimbra -- UID/MAT/00324/2013, funded by the Portuguese Government through FCT/MEC and co-funded by the European Regional Development Fund through the Partnership Agreement PT2020.}

\authorone[CMUC, University of Coimbra, Portugal]{Idir Arab} 
\addressone{Department of Mathematics, PO Box 3008, EC Santa Cruz, Coimbra, Portugal} 
\authortwo[University of Central Lancashire, Cyprus]{Milto Hadjikyriakou}
\addresstwo{School of Sciences, 12-14 University Avenue, Pyla, 7080 Larnaka, Cyprus}
\authorthree[CMUC, Department of Mathematics, University of Coimbra, Portugal]{Paulo Eduardo Oliveira}
\addressthree{Department of Mathematics, PO Box 3008, EC Santa Cruz, Coimbra, Portugal}

\begin{abstract}
We study failure rate monotonicity and generalized convex transform stochastic ordering properties of random variables, with a concern on applications. We are especially interested in the effect of a tail weight iteration procedure to define distributions, which is equivalent to the characterization of moments of the residual lifetime at a given instant. For the monotonicity properties, we are mainly concerned with hereditary properties with respect to the iteration procedure providing counter-examples showing either that the hereditary property does not hold or that inverse implications are not true. For the stochastic ordering, we introduce a new criterium, based on the analysis of the sign variation of a suitable function. This criterium is then applied to prove ageing properties of parallel systems formed with components that have exponentially distributed lifetimes.
\end{abstract}



\section{Introduction}
According to Barlow and Proschan~\cite{BP-81}, one of the most important aims of reliability theory is to provide researchers with all the necessary tools to understand, estimate and optimize the life span and failure distributions of systems and their components. In reliability theory, ageing is defined as a phenomenon of increasing risk of failure with the passage of time. If the risk of failure is not increasing with age (the ``old is as good as new'' principle), then there is no ageing in terms of reliability theory, even if the calendar age of a system is increasing. Thus, the regular and progressive changes over time do not constitute ageing unless they produce some deleterious outcome (failures). Rausand and H\o{}yland~\cite{RH-04}, define failure as the event that makes the system to behave differently than what it is desired and expected. ``Positive ageing'' can be identified in cases where the residual lifetime tends to decrease with increasing age of the system. ``Negative ageing'' (also known as ``beneficial ageing'') has the exact opposite effect, but this is a less common situation and has attracted significantly less research interest.

\medskip

Ageing properties can be employed in order to define different classes of life distributions. Note that the exponential distribution is a member of almost every class, exactly because of its memoryless property.  Lifetime distributions can be characterized by their reliability function, the conditional survival function, their failure rate or their expected value of residual life. These quantities are used to express different notions of ageing also known as reliability classes. For example, distributions that have either increasing failure rate (IFR) or decreasing failure rate (DFR) have been studied by various researchers, while other notions such ``increasing failure rate on average'' (IFRA), ``new better than used'' (NBU), ``new worse than used'' (NWU), ``new better than used in expectation'' (NBUE), ``new worse than used in expectation'' (NWUE) and ``decreasing mean residual life'' (DMRL) have also attracted a lot of attention.

\medskip

The interesting properties of these ageing classes include preservation or closure properties of a given class under the formation of coherent systems of independent components, under convolution or mixture. It is also important to be able to provide reliability bounds and moment inequalities and test exponentiality against other lifetime distributions. Properties of IFR and DFR have been studied by Barlow and Proschan~\cite{BP-81}, Patel~\cite{Pt-83} while important results for IFRA can be found in Barlow and Proschan~\cite{BP-81}, Sengupta~\cite{Sp-94} and El-Bassiouny~\cite{ElB-03}. Abouammoh and El-Neweihi~\cite{AE-N86} showed that the NBU class is closed under formation of parallel systems of independent and identically distributed components while Barlow and Proschan~\cite{BP-81} provided probability bounds for NBU, NWU, NBUE and NWUE. Chen~\cite{CH-89} showed that the distributions of these classes may be characterized through certain properties of the corresponding renewal functions, Cheng and He~\cite{CH-89} studied the reliability bounds on NBUE and NWUE classes while Cheng and Lam~\cite{CL-02} obtained reliability bounds on NBUE from the first two known moments. Bryson and Siddiqui~\cite{BS-69} proved that IFR (DFR) implies DMRL (IMRL), Abouammoh and El-Neweihi~\cite{AE-N86} proved that DMRL classes are closed under the formation of parallel systems, while Abu-Youssef~\cite{AY-02} derived a moment inequality that was used by the author to derive a test for testing exponentiality against DMRL (IMRL).

\medskip

Another important aspect in the study of lifetime distributions is their order relations. These usually define partial orderings which establish the comparison between two lifetime variables in terms of their failure rates, density functions, survival functions, mean residual lives or other ageing characteristics. Ageing classes can often be characterized by some partial orderings. Barlow and Proschan~\cite{BP-81} proved that IFR and IFRA classes are characterized by some specific choice of ``convex ordering'' and ``star-shaped ordering'' respectively.  Partial ordering of lifetime distributions has been studied extensively by various authors (see for example, Desphande et al.~\cite{Des86}, Kochar and Wiens~\cite{KW-87}, Singh~\cite{SH-89}, Fagiuoli and Pellerey~\cite{FP93}, Shaked and Shanthikumar~\cite{SS07}) because of their applicability in a wide spectrum of different fields such as econometrics (Whitmore~\cite{Wh-70}), reliability (Barlow and Proschan~\cite{BP-81}), queues (Stoyan~\cite{Sy-83}) and other stochastic processes (Ross~\cite{Rss-83}). Singh and Jain~\cite{SJ-89} and Fagiuoli and Pellerey~\cite{FP93} have proposed an application to stochastic comparison between two devices that are subjected to Poisson shock models.

\medskip

Over the last decades there is an increasing interest in generalized partial orderings and several generalizations can be found in the literature. Some of these new ordering notions led to the creation of new ageing classes. Averous and Meste~\cite{AM89} and later Fagiuoli and Pellerey~\cite{FP93} introduced new concepts of partial stochastic ordering, namely $s-$FR, $s-$ST, $s-$CV, $s-$CX and $s-$SFR. For the case where $s=1$ or $s=2$ the new orderings is reduced to well-known stochastic orders. The authors provide relations between the new ordering concepts and classical partial orders and they also give the definitions of the related classes of life distributions. Nanda et al.~\cite{ASOK} introduced new generalized partial orderings, particularly the $\sIFR{s}$, $\sIFRA{s}$, $s-$NBU, $s-$NBUFR and $s-$NBAFR orderings. In their paper they also provide some equivalent representations for each ordering and they also discuss inter-relations among these orderings. Again for $s=1, 2, 3$ some of these new orderings are equivalent to already known partial stochastic orders. Despite the fact that for higher values of $s$ these partial orderings may not have clear and meaningful applications, their mathematical nicety and the fact that they unify existing results, make their study very interesting. Nevertheless, one motivation for these extensions can be found in Loh~\cite{Lh-84} where different types of generalized partial orderings were used for testing for discrepancies in the tails of symmetric distributions.

\medskip

Researchers are often interested in comparing the skewness of two distributions. van Zwet~\cite{VZ-70} introduced a new skewness order, the so-called convex transform order. In reliability theory the particular order is used to capture the fact that one distribution is more IFR-increasing failure rate than another distribution. Kochar and Xu~\cite{KX09} proved that a parallel system with heterogeneous exponential component lifetimes is more skewed (according to the IFR order) than the system with independent and identically distributed exponential components. In other words, they proved that a parallel system with homogeneous exponential components, ages faster than a system with heterogeneous exponential components in the sense of the ``smaller in IFR'' property. Note that in what follows the IFR ordering will be denoted by $\sIFR{1}$, following the notation introduced by the references mentioned above. Many authors have studied orderings of such systems when the parameters of the exponential distributions satisfy certain restrictions (see for example Dykstra et al.~\cite{DKR-97}, Khaledi and Kochar~\cite{KhK-00}, Kochar and Xu~\cite{KX-07a,KX-07b}, among many other authors). Recently, a number of researchers have also studied the case where the exponential distribution is substituted with some generalized versions (see, for example, Balakrishnan et al.~\cite{BHM-15}, Bashkar et al.~\cite{BTA-17}).

\medskip

In this paper, we study some properties of lifetimes that are either $\sIFR{s}$ or $\sIFRA{s}$ and at the same time we are interested in constructing criteria that will enable us to identify whether specific lifetime distributions are ordered via the $\sIFR{s}$ order. One of the main results of this work is that although in general, the $\sIFR{s}$ (or the $\sIFRA{s}$) ordering is not an inherited trait of distributions, Theorem 3.1 of Kochar and Xu~\cite{KX09} is verified for the $\sIFR{s}$ ordering where $s$ can be any positive integer.

\medskip

The paper is structured as follows: in Section 2 we provide some definitions and results that will be useful for the rest of the paper while Sections 3 and 4 refer to properties of distributions that are either $\sIFR{s}$ or $\sIFRA{s}$. In Section 5 we will present an example of distributions that proves that two stochastic orders that were reported in the literature as equivalent are in fact two different concepts. The main results of the paper are concentrated in Sections 6 and 7. Particularly, in Section 6  we provide a new criterium for the $\sIFR{s}$ ordering via $\sIFRA{s}$ order and in Section 7 this new criterium is used to prove ageing properties of parallel systems formed with components that have exponentially distributed lifetimes.


\section{Preliminaries}
We recall here the basic definitions and representations about the tail-weight iterated distributions. These iterated distributions were introduced by Averous and Meste~\cite{AM89} and initially studied by Fagiuoli and Pellerey~\cite{FP93}. Let $X$ be a nonnegative random variable with density function $f_X$, distribution function $F_X$, and tail function $\overline{F}_X=1-F_X$.
\begin{defn}
\label{def:s-iter}
For each $x\geq 0$, define
\begin{equation}
\TXs{0}(x)=f_{X}(x)\quad \mbox{and}\quad \muXs{0}=\int_0^\infty \TXs{0}(t)\,dt=1.
\end{equation}
For each $s\geq 1$, define the $s-$iterated distribution $T_{X,s}$ by its tail $\TXs{s}=1-T_{X,s}$ as follows:
\begin{equation}
\TXs{s}(x)=\frac{1}{\muXs{s-1}}\int_{x}^\infty\TXs{s-1}(t)\,dt
\quad \mbox{where}\quad \muXs{s}=\int_0^\infty \overline{T}_{X,s}(t)\,dt.
\end{equation}
Moreover, we extend the domain of definition of each $\TXs{s}$ by defining $\TXs{s}(x)=1$ for $x<0$.
\end{defn}
The distribution $\TXs{2}$ is also known as the equilibrium distribution of $X$, and plays an important role in ageing relations (see, for example, Chatterjee and Mukherjee~\cite{CM01}) and in renewal theory (see Cox~\cite{Cox62}). Hence, the iteration process above defines, for each $s\geq 1$, $\TXs{s}$ as the equilibrium distribution of a random variable with tail $\TXs{s-1}$. Although the definitions are introduced in a recursive way, a closed form representation for the iterated distributions is available.
\begin{lem}[Lemma 2 and Remark 3 in Arab and Oliveira~\cite{AO18}]
\label{Simple T_s}
The tails $\TXs{s}$ may be represented as
\begin{equation}
\label{eq:simpleTs}
\TXs{s}(x)=\frac{1}{\dE X^{s-1}}\int_x^\infty f_{X}(t)(t-x)^{s-1}\,dt.
\end{equation}
The $s-$iterated distribution moments are given by
\begin{equation}
\label{eq:moms}
\muXs{s}=\frac{1}{s}\frac{\dE X^{s}}{\dE X^{s-1}}.
\end{equation}
\end{lem}
Note that (\ref{eq:simpleTs}) may be rewritten as
\begin{equation}
\label{eq:rep.mom}
\TXs{s}(x)=\frac{1}{\dE X^{s-1}}\dE(X-x)_+^{s-1},
\end{equation}
where $(X-x)_+=\max(0,X-x)$ is the residual lifetime at age $x$. Therefore, the $s-$iterated distribution may be interpreted as the normalized survival moment of order $s-1$.

\medskip

One of the most simple and common ageing notion is defined through the monotonicity of the failure rate function of a distribution $\frac{f_X(x)}{1-F_X(x)}=\frac{\TXs{0}(x)}{\TXs{1}(x)}$. The direct verification of this monotonicity is, in general, not a simple task, as for many distributions the tail does not have an explicit closed representation or, at least, not a manageable one. Having defined iterated distributions, it becomes natural to proceed likewise with respect to the failure rate functions, as defined in Nanda et al.~\cite{ASOK} and also studied in Arab and Oliveira \cite{AO18}.
\begin{defn}
\label{def:s-fail}
For each $s\geq 1$ and $x\geq 0$, define the $s-$iterated failure rate function as
\begin{equation*}
r_{X,s}(x)=\frac{\overline{T}_{X,s-1}(x)}{\int_x^\infty\overline{T}_{X,s-1}(t)\,dt}
=\frac{\overline{T}_{X,s-1}(x)}{\widetilde{\mu}_{X,s-1}\overline{T}_{X,s}(x)}.
\end{equation*}
\end{defn}
It is obvious that for $s=1$ we find the failure rate of $X$,
$r_{X,1}(x)=\frac{f_{X}(x)}{\overline{F}_{X}(x)}$, hence the monotonicity of the failure rate is expressed as the monotonicity of $r_{X,1}$. We may extend this monotonicity notion by considering the $s-$iterated distribution, as done in Averous and Meste~\cite{AM89}, Fagiuoli and Pellerey~\cite{FP93}, Nanda et al.~\cite{ASOK}, among many other references.
\begin{defn}
\label{def:age}
For $s=1,2,\ldots$, the nonnegative random variable $X$ is said to be
\begin{enumerate}
\item
\sIFR{s} (resp. \sDFR{s}) if $r_{X,s}$ is increasing (resp. decreasing) for $x\geq 0$.
\item
\sIFRA{s} (resp. $\sDFRA{s}$) if $\frac{1}{x}\int_0^x r_{X,s}(t)\,dt$ is increasing (resp. decreasing) for $x>0$.
\end{enumerate}
\end{defn}
The above mentioned references introduce a few other monotonicity notions, but we refer only the ones to be addressed in the present paper. Remark that it follows easily from the definition above  that the $\sIFR{s}$ monotonicity of a random variable $X$ implies that the variable is also $\sIFRA{s}$.


%

\medskip

We introduce next the order relations to be addressed.
\begin{defn}
\label{DEF S-IFR}
Let $\mathcal{F}$ denote the family of distributions functions such that $F(0)=0$, $X$ and $Y$ be nonnegative random variables with distribution functions $F_X,F_Y\in\mathcal{F}$, and $s\geq 1$ an integer.
    \begin{enumerate}
    \item
    The random variable $X$ (or its distribution $F_X$) is said smaller than $Y$ (or its distribution $F_Y$) in \sIFR{s} order, and we write $X\leq_{\sIFR{s}}Y$, or equivalently, $F_X\leq_{\sIFR{s}}F_Y$, if $c_s(x)=\TYs{s}^{-1}(\TXs{s}(x))$ is convex.
    \item
    The random variable $X$ (or its distribution $F_X$) is said smaller than $Y$ (or its distribution $F_Y$) in \sIFRA{s} order, and we write $X\leq_{\sIFRA{s}}Y$, or equivalently, $F_X\leq_{\sIFRA{s}}F_Y$, if $t_s(x)=\frac{1}{x}c_s(x)$ is increasing (this is also known as $c_s(x)$ being star-shaped).
    \end{enumerate}
\end{defn}
Fagiuoli and Pellerey~\cite{FP93} and Nanda et al.~\cite{ASOK} concentrated on establishing relations between the ordering notions defined. It is useful to note that these order relations define partial order relations in the equivalence classes of $\mathcal{F}$ corresponding to the equivalence relation $F\sim G$ defined by $F(x)=G(kx)$, for some $k>0$. In case of families of distributions that have a scale parameter, this allows to choose the parameter in the most convenient way.

\smallskip

The exponential distribution plays an important role when dealing with ageing notions. Besides being a fixed point with respect to the iteration procedure, the comparability with the exponential, either in the $\sIFR{s}$ or $\sIFRA{s}$ sense, is equivalent to the $\sIFR{s}$ or $\sIFRA{s}$ monotonicity as proved by Nanda et al.~\cite{ASOK} (see Theorems~3.2 and 4.3)
\begin{thm}
\label{thm:expon}
Let $X$ be a random variable with distribution function $F_X\in\mathcal{F}$ and $Y$ with exponential distribution. 
\begin{enumerate}
\item
$X\leq_{\sIFR{s}}Y$ (resp., $Y\leq_{\sIFR{s}}X$) if and only if $X$ is \sIFR{s} (resp., $X$ is \sDFR{s}).
\item
$X\leq_{\sIFRA{s}}Y$ (resp., $Y\leq_{\sIFRA{s}}X$) if and only if $X$ is \sIFRA{s} (resp., $X$ is \sDFRA{s}).
\end{enumerate}
\end{thm}
As an immediate consequence of the above, we have the following comparison results.
\begin{cor}
Let $X$ and $Y$ be random variables with distribution functions $F_X,F_Y\in\mathcal{F}$, and $s\geq 1$ an integer. If $X$ is $\sIFR{s}$ and $Y$ is $\sDFR{s}$, then $X\leq_{\sIFR{s}}Y$. The same holds replacing {\rm IFR} and {\rm DFR} by {\rm IFRA} and {\rm DFRA}, respectively.
\end{cor}
A general characterization of the above order relations is given below (see Propositions~3.1 and 4.1 in Nanda et al.~\cite{ASOK}).
\begin{thm}
\label{convexity-equivalence}
Let $X$ and $Y$ be random variables with distribution functions $F_X,F_Y\in\mathcal{F}$.
    \begin{enumerate}
    \item
    $X\leq_{\sIFR{s}}Y$ if and only if for any real numbers $a$ and $b$, $\TYs{s}(x)-\TXs{s}(ax+b)$ changes sign at most twice, and if the change of signs occurs twice, it is in the order ``$+,-,+$'', as $x$ traverses from $0$ to $+\infty$.
    \item
    $X\leq_{\sIFRA{s}}Y$ if and only if for any real number $a$, $\TYs{s}(x)-\TXs{s}(ax)$ changes sign at most twice, and if the change of signs occurs twice, it is in the order ``$-,+$'', as $x$ traverses from $0$ to $+\infty$.
    \end{enumerate}

\begin{rem}
\label{AOalpha}
As mentioned in Remark 25 in Arab and Oliveira~\cite{AO18}, it is enough to verify the above characterizations only for $a>0$.
\end{rem}
\end{thm}
The above characterization requires explicit expressions of the tails of the iterated distributions, which are often not available. Computationally tractable characterizations to decide about the actual comparison of general distributions were studied in Arab and Oliveira~\cite{AO18}. We quote the characterization proved in Theorem~27 and Corollary~29 in \cite{AO18}.
\begin{thm}
\label{AOmain}
Let $X$ and $Y$ be random variables with absolutely continuous distributions with densities $f_X$ and $f_Y$ and distribution functions $F_X,F_Y\in\mathcal{F}$, respectively. If, for every constants $a>0$ and $b\in\mathbb{R}$, either of the functions,
$$
H_s(x)=\frac{1}{\dE Y^{s-1}}f_Y(x)-\frac{a^{s}}{\dE X^{s-1}}f_X(ax+b)
$$
or
$$
H_{s-1}(x)=\frac{1}{\dE Y^{s-1}}\overline{F}_Y(x)-\frac{a^{s-1}}{\dE X^{s-1}}\overline{F}_X(ax+b).
$$
changes sign at most twice when $x$ traverses from $0$ to $+\infty$, and if the change of sign occurs twice, it is in the order ``$+,-,+$'', then $F_X\leq_{\sIFR{s}}F_Y$.

The functions $H_s$ and $H_{s-1}$ may, respectively, be replaced by
$$
P_s(x)=\log f_Y(x)-\log f_X(ax+b)+\log\frac{\dE X^{s-1}}{a^{s}\dE Y^{s-1}},
$$
and
$$
P_{s-1}(x)=\log \overline{F}_Y(x)-\log\overline{F}_X(ax+b)+\log\frac{\dE X^{s-1}}{a^{s-1}\dE Y^{s-1}}.
$$
\end{thm}
The next statement provides a criterium to verify the $\sIFRA{s}$ order relation. We do not include a proof, as this follows reproducing the arguments presented in Arab and Oliveira~\cite{AO18} for the proof of the previous result.
\begin{thm}
\label{AOmain1}
Let $X$ and $Y$ be random variables with absolutely continuous distributions with densities $f_X$ and $f_Y$ and distribution functions $F_X,F_Y\in\mathcal{F}$, respectively. If, for every $a>0$ and $b=0$, either of the functions $H_s(x)$ or $H_{s-1}(x)$ changes sign at most once when $x$ traverses from $0$ to $+\infty$, and if the change of sign occurs, it is in the order ``$-,+$'', then $F_X\leq_{\sIFRA{s}}F_Y$. The functions $H_s$ and $H_{s-1}$ may, respectively, be replaced by $P_s(x)$ and $P_{s-1}(x)$.
\end{thm}

As indicated in Theorem~\ref{convexity-equivalence}, the control of the sign variation of $\TYs{s}(x)-\TXs{s}(ax+b)$ is crucial to characterize the $\sIFR{s}$ and $\sIFRA{s}$ ordering. This function is obtained after integration of $H_s$ or $H_{s-1}$ defined in Theorems~\ref{AOmain} and \ref{AOmain1}. Below, we quote a result about sign variation after integration, used in Arab and Oliveira~\cite{AO18} for the proof of Theorem~\ref{AOmain}, that is an important tool for our results below.
\begin{lem}[Lemma~26 in Arab and Oliveira~\cite{AO18}]
	\label{sign-integral}
	Let $f$ and $g$ be two real-valued functions defined on $[0,\infty)$ such that,
	$$
	g(x)=\int_{x}^\infty f(t)\,dt.
	$$
	Assume that, as $x$ traverses from $0$ to $+\infty$, $f(x)$ changes sign in one of the following orders ``$-,+$'' or ``$+,-$'' or ``$+,-,+$'' or ``$-,+,-,+$''. Then $g(x)$, as $x$ traverses from $0$ to $\infty$, has sign variation equal to every possible final part of the sign variation of $f(x)$.
\end{lem}


The lifetime of parallel systems is expressed as the maximum of the lifetimes of each component. When these component have exponentially distributed lifetimes, the distribution functions of the system's lifetime is expressed as a linear combinations of exponential terms. Later, it will be important to be able to count and localize the roots of such expressions. The following result will play an important role on this aspect.
\begin{thm}[Theorem 1 in Shestopaloff~\cite{She11}]
\label{thm:zeros}
Let $n\geq 0$, $p_0>p_1>\cdots>p_n>0$, and $\alpha_j\ne0$, $j=0,1,\ldots,n$, be real numbers. Then the function $f(t)=\sum_{j=0}^n\alpha_jp^t_j$ has no real zeros if $n=0$, and for $n\geq 1$ has at most as many real zeros as there are sign changes in the sequence of coefficients $\alpha_0,\alpha_1,\alpha_2,\ldots,\alpha_n$.
\end{thm}

\section{Hereditary monotonicity properties}
A common feature about iterated monotonicity properties is an hereditary with respect to the iteration parameter. However, the hereditary property does not hold for all the order relations defined, as we will be showing below by an example. We quote first the hereditary property for monotonicity of the iterated failure rate.
\begin{lem}
\label{IFR}
Let $X$ be a nonnegative random variable. For every integer $s\geq 1$, the following relations hold.
\begin{itemize}
  \item[a)] If $X$ is \sIFR{s}, then $X$ is \sIFR{(s+1)}.
  \item[b)] If $X$ is \sDFR{s}, then $X$ is \sDFR{(s+1)}.
\end{itemize}
\end{lem}
This result is included in Theorem~2 in Navarro and Hernandez~\cite{NH04}. It implies that, for most distributions, it is enough to verify the \sIFR{1} or the \sDFR{1} property.
Exhibiting distributions that do not have lower iterated monotonicity but verify it after a few iteration steps, usually requiring a suitable modification of known families of distributions. 
Such an example, using fattened tail Pareto distributions, was given in Example~9 in Arab and Oliveira~\cite{AO18}. 
This hints a way to construct distributions with failure rates that become monotone only after a few iteration steps. The example below shows the same effect for a distribution that is IFR, instead. 
\begin{ex}
\label{ex:s-s+1IFR}
Let $X$ be a nonnegative random variable with density function $f(x)= \frac{(x^2+c)e^{-x}}{c+2}$. It is easily verified that
\begin{equation*}
\TXs{1}(x)=\frac{(x^2+2x+2+c)e^{-x}}{c+2}, \qquad \TXs{2}(x)=\frac{(x^2+4x+6+c)e^{-x}}{6+c},
\end{equation*}
and
\begin{equation*}
  r_{X,1}(x)=\frac{f(x)}{\TXs{1}(x)}=\frac{x^2+c}{x^2+2x+2+c}, \qquad r_{X,2}(x)=\frac{6+c}{2+c}\!\times\!\frac{x^2+2x+2+c}{x^2+4x+6+c}.
\end{equation*}
Differentiating, we find
\begin{equation*}
  r_{X,1}^{\prime}(x)=\frac{2x^2+4x-2c}{\left(x^2+2x+2+c\right)^2}, \qquad r_{X,2}^{\prime}(x)=\frac{6+c}{2+c}\!\times\!\frac{2x^2+8x+4-2c}{\left(x^2+4x+6+c\right)^2}.
\end{equation*}
By choosing $c\in(0,2)$ we obtain $r_{X,1}$ that starts decreasing and eventually becomes increasing, while $r_{X,2}$ is increasing. That is, $X$ is not $\sIFR{1}$ but is $\sIFR{2}$. Moreover, it is also easy to verify that $X$ is not $\sIFRA{1}$.
\end{ex}
Let us now look at the hereditary property concerning the $\sIFRA{s}$ monotonicity, to show that the situation is quite different from what happens with the $\sIFR{s}$ monotonicity.

\begin{prop}
\label{prop:ex1}
Let $Y_1$ and $Y_2$ be independent exponential random variables with mean 1 and $1/\lambda$, respectively, where $\lambda\neq 1$, and define $Y=\max(Y_1,Y_2)$. Then $Y$ is $\sIFRA{1}$, but it is neither $\sIFRA{2}$ nor $\sDFRA{2}$. Moreover, there exists $s_0>2$ such that $Y$ is $\sDFR{s}$ for every $s\geq s_0$.
\end{prop}
\begin{proof}
Remark first that
$$
\TYs{s}(x)=\frac{1}{c(s,\lambda)}\left(e^{-x}+\frac{e^{-\lambda x}}{\lambda^{s-1}}-\frac{e^{-(\lambda+1) x}}{(\lambda+1)^{s-1}}\right),
$$
where $c(s,\lambda)=1+\frac{1}{\lambda^{s-1}}-\frac{1}{(\lambda+1)^{s-1}}$.
To prove that $Y$ is $\sIFRA{1}$, we need to verify that $-t_1(x)=\frac{\log(\TYs{1}(x))}{x}$ is decreasing. Taking into account Theorem~\ref{convexity-equivalence}, we need to prove that $H(x)=e^{-x}+e^{-\lambda x}-e^{-(1+\lambda) x}-e^{-ax}$ changes sign at most once in the order ``$+,-$'' for every $a>0$ (note that we are here interested in proving the function is decreasing, while Theorem~\ref{convexity-equivalence} characterized increasingness). Moreover, remark that, for every $x\geq 0$, $\TYs{1}(x)\geq e^{-x}$, so it is enough to consider $a<1$.
%
%
%
%
Hence, the sign pattern of the coefficients means that, according to Theorem~\ref{thm:zeros}, $H$ has at most two real roots. Moreover, we have that $\lim_{x\rightarrow+\infty}H(x)=0^-$, $H(0)=0$, and $H^\prime(0)=a>0$, so it follows that the second root does exist and is positive, consequently the sign variation of $H(x)$ is ``$+,-$''. Therefore, we have proved that $Y$ is $\sIFRA{1}$.

To prove the second statement, we verify that $t_2(x)=\frac{-\log(\TYs{2}(x))}{x}$ is not monotone. Indeed, we have
$$\lim_{x\rightarrow0}t_2(x)=\frac{1}{1+\frac{1}{\lambda}-\frac{1}{1+\lambda}}<1,
\qquad\mbox{and}\qquad
\lim_{x\rightarrow+\infty}t_2(x)=1.
$$
We verify now that the equation $t_2(x)=1$ has one positive solution. Rewrite this as
$$
t_2(x)=1\quad\Leftrightarrow\quad P(x)=\left(\frac{1}{\lambda+1}-\frac{1}{\lambda}\right)e^{-x}+\frac{e^{-\lambda x}}{\lambda}-\frac{e^{-(1+\lambda)x}}{1+\lambda}=0.
$$
Again from Theorem~\ref{thm:zeros}, $P$ has at most two real roots. As $P(0)=0$, $P^\prime(0)>\frac{1}{\lambda}-\frac{1}{\lambda+1}>0$ and $\lim_{x\rightarrow+\infty} P(x)=0^-$, there exists a strictly positive second root.
%

For the final statement, we want to prove the monotonicity of
$$
r_{Y,s}(x)=\frac{\TYs{s-1}(x)}{\widetilde{\mu}_{Y,s-1}\TYs{s}(x)}
$$
which coincides with the monotonicity of
$$
N(x)=
\frac{e^{-x}+\frac{e^{-\lambda x}}{\lambda^{s-2}}-\frac{e^{-(\lambda+1) x}}{(\lambda+1)^{s-2}}}{e^{-x}+\frac{e^{-\lambda x}}{\lambda^{s-1}}-\frac{e^{-(\lambda+1) x}}{(\lambda+1)^{s-1}}}.
$$
We look at the numerator of $N^\prime(x)$, which after some algebraic manipulation, may be written as
$$
Q(x)=
-\frac{(\lambda-1)^2}{\lambda^{s-1}}e^{-(\lambda+1)x}+\frac{\lambda^2}{(\lambda+1)^{s-1}}e^{-(\lambda+2)x}
+\frac{1}{(\lambda^2+\lambda)^{s-1}}e^{-(2\lambda+1)x}
$$
Of course, the sign of $N^\prime(x)$ coincides with the sign of $Q(x)$. Notice that if $\lambda>1$, we have $\lambda+1<\lambda+2<2\lambda+1$, while the two last terms interchange when $\lambda<1$. Hence, it follows from Theorem~\ref{thm:zeros} that $Q$ has, at most, one real root. Moreover, $\lim_{x\rightarrow+\infty}Q(x)=0^-$. Therefore, if $Q(0)>0$ the sign variation of $Q$ in $(0,+\infty)$ is ``$+,-$'', and if $Q(0)<0$ the sign variation is ``$-$''. We have that
$$
Q(0)=\frac{\lambda^{s+1}+1-(\lambda-1)^2(1+\lambda)^{s-1}}{(\lambda^2+\lambda)^{s-1}},
$$
and this, as function of $s$, will eventually become negative as the numerator has a negative coefficient for $\lambda^s$, the largest power in that expression.
\end{proof}
The following is an immediate consequence of Example~\ref{ex:s-s+1IFR} and Proposition~\ref{prop:ex1}.
\begin{cor}
The $\sIFRA{s}$ monotonicity does not have the hereditary property.
\end{cor}

\section{Simple failure rate monotonicity properties} 
This section presents simple properties of IFR or DFR distributions that are not of hereditary nature. We first highlight som improvement on classical moment bounds that may be derived from the iterated failure rate monotonicity.
\begin{prop}
\label{prop:holder}
Let $X$ be a random variable with distribution function $F_X\in\mathcal{F}$ and density function $f_X$, and let $s>3$.
\begin{enumerate}
\item
If $X$ is $\sIFR{s}$, then
$$
\left(1-\frac{1}{s-1}\right)\dE(X-x)_+^{s-3}\dE(X-x)_+^{s-1}\leq\left(\dE(X-x)_+^{s-2}\right)^2\leq\dE(X-x)_+^{s-3}\dE(X-x)_+^{s-1}.
$$
\item
If $X$ is $\sDFR{s}$, then
$$
\left(\dE(X-x)_+^{s-2}\right)^2\leq\left(1-\frac{1}{s-1}\right)\dE(X-x)_+^{s-3}\dE(X-x)_+^{s-1}.
$$
\end{enumerate}
\end{prop}
\begin{proof}
A direct application of  H\"older inequality justifies the upper bound in $\sIFR{s}$ case. Both the lower bound in the $\sIFR{s}$, and the upper bound in the $\sDFR{s}$ case follow by requiring the appropriate sign on the numerator of $r_{X,s}^\prime$ and taking into account (\ref{eq:moms}) and (\ref{eq:rep.mom}).
\end{proof}
\begin{rem}
Note that the previous result, in the case of DFR distributions, provides a bound for the $s-2$ moment of the residual life at age $x$ sharper than what is given by the H\"older inequality. For the case of IFR distributions, Proposition~\ref{prop:holder} together with H\"older inequality gives a sharp interval for the $s-2$ moment of the residual life at age $x$.
\end{rem}


We now have a look at the iterated failure rate properties of parallel systems. The lifetime of such a system is expressed mathematically as the maximum of the lifetimes of each one of the components, has already been used to provide an example about the nonhereditary of the IFRA monotonicity (see Proposition~\ref{prop:ex1}). We recall here a well known property about the monotonicity of parallel systems and derive a few simple consequences.
\begin{prop}
Let $X_1,\ldots,X_n$ be independent and identically distributed $\sIFR{1}$ random variables, with distribution function $F\in\mathcal{F}$ and density function $f$. Then $X_{(n)}=\max(X_1,\ldots,X_n)$ is $\sIFR{s}$, for every $s\geq 1$.
\end{prop}
\begin{proof}
Taking into account Lemma~\ref{IFR}, it is enough to verify that is $X_{(n)}$ is $\sIFR{1}$. Writing $r_{X_{(n)},1}(x)=\frac{nF^{n-1}(x)}{1+F(x)+\cdots+F^{n-1}(x)}\frac{f(x)}{1-F(x)}$, the conclusion is immediate.
\end{proof}
\begin{rem}
Although the result presented above i.e. the property that parallel systems of identical $\sIFR{1}$ units are also $\sIFR{1}$ is a known result (see for example~\cite{BP-81}), we present its proof for the sake of completeness. Note that, to the best of our knowledge, this is a new approach for the proof of the particular property. An alternative proof for $n=2$ can be found in Example~A.11 in Marshall and Olkin~\cite{MO07}.
\end{rem}
An easy consequence follows if we form the parallel system with components after a few iteration steps.
\begin{cor}
Let $X$ be $\sIFR{s}$, for some $s\geq1$, random variable, with distribution function $F\in\mathcal{F}$ and density function $f$. Let $Y_{(n)} = \max(Y_1,\ldots,Y_n)$ where $Y_i$ are independent with tail function $\TXs{s}(x)$. Then $Y_{(n)}$ is $\sIFR{s}$, for every $s\geq 1$.
\end{cor}

A related result was proved in Theorem 2.2 in Abouammoh and El-Neweihi~\cite{AE-N86}, that we quote here presented in a slightly more general wording.
\begin{prop}
\label{parallel}
Let $X_1,\ldots,X_n$ be independent and identically distributed $\sIFR{s}$, for some $s\geq 2$, random variables with distribution function $F\in\mathcal{F}$. Let $Y_{(n)} = \max(Y_1,\ldots,Y_n)$ where $Y_i$ are independent with tails $\TXs{s-1}(x)$. Then $Y_{(n)}$ is $\sIFR{s}$.
\end{prop}
The original statement by Abouammoh and El-Neweihi~\cite{AE-N86} considers only the case where $s=2$. The above version follows immediately by remembering the hereditary of the $\sIFR{s}$ monotonicity. Both results prove the iterated monotonicity of maxima based on distributions constructed after some iteration steps. The statement in Proposition~\ref{parallel} has a more straightforward practical interpretation.
\section{Nonhereditary of the $\sIFR{s}$ ordering}
\label{sec:false}
We now have a look at hereditary properties of the $\sIFR{s}$ ordering. We shall prove that, opposite to what happens with the $\sIFR{s}$ monotonicity, the ordering relation is not an hereditary property.
For the discussion, we need to recall one more stochastic order relation (see Section~4.B.2 in Shaked, Shanthikumar~\cite{SS07}).
\begin{defn}
Let $X$ and $Y$ be random variables with distribution functions $F_X,F_Y\in\mathcal{F}$. The random variable $X$ is said to be more {\rm DMRL} than Y, and we write $X\leq_{{\rm DMRL}}Y$, if $d(x)=\frac{\TYs{2}(\TYs{1}^{-1}(x))}{\TXs{2}(\TXs{1}^{-1}(x))}$ is decreasing.
\end{defn}
The following relation with failure rate order holds.
\begin{thm}[Theorem~4.B.20 in Shaked, Shanthikumar~\cite{SS07}]
Let $X$ and $Y$ be random variables with distribution functions $F_X,F_Y\in\mathcal{F}$. If $X\leq_{\sIFR{1}}Y$, then $X\leq_{{\rm DMRL}}Y$.
\end{thm}
Nanda et al.~\cite{ASOK}, mention in their Remark~3.1, without proof, that the DMRL order is equivalent to the $\sIFR{2}$ order. An immediate consequence of Nanda et al.~\cite{ASOK} remark is that if $X\leq_{\sIFR{1}}Y$ then $X\leq_{\sIFR{s}}Y$ for any $s\geq 2$.
%
%
Indeed, once proved that $X\leq_{\sIFR{1}}Y$, Theorem~4.B.20 in Shaked and Shanthikumar~\cite{SS07}, implies that $X\leq_{{\rm DMRL}}Y$, hence, according to Nanda et al.~\cite{ASOK} remark, $X\leq_{\sIFR{2}}Y$. If we define now $X^\ast_2$ with tail function $\TXs{2}$, and $Y^\ast_2$ with tail function $\TYs{2}$, the previous order relation means that $X^\ast_2\leq_{\sIFR{1}}Y^\ast_2$. Therefore, iterating once again, and applying Theorem~4.B.20 from Shaked and Shanthikumar~\cite{SS07} and the mentioned remark, it follows that $X^\ast_2\leq_{\sIFR{2}}Y^\ast_2$, which is just a rewriting for $X\leq_{\sIFR{3}}Y$. Repeating the above construction, it would follow that $X\leq_{\sIFR{s}}Y$, for every $s\geq 1$. However, the equivalence mentioned in Remark~3.1 of Nanda et al.~\cite{ASOK} is, in general, not true. We can prove the stated equivalence only when one of the random variables is exponentially distributed. The construction of a counter-example for the general result requires a very careful choice of distribution functions, as described below in Proposition~\ref{prop:counter}.
\begin{prop}
Let $X$ be a random variable with distribution function $F_X\in\mathcal{F}$ and $Y$ a random variable with exponential distribution. Then $X\leq_{\sIFR{2}}Y$ if and only if $X\leq_{{\rm DMRL}}Y$.
\end{prop}
\begin{proof}
Taking into account the comments after Definition~\ref{DEF S-IFR}, it is enough to consider the case where $Y$ has mean 1. Then we have that $\TYs{1}(x)=\TYs{2}=e^{-x}$. Therefore, $X\leq_{{\rm DMRL}}Y$ is equivalent to $d(x)=\frac{\TYs{2}(\TYs{1}^{-1}(x))}{\TXs{2}(\TXs{1}^{-1}(x))}=\frac{x}{\TXs{2}(\TXs{1}^{-1}(x))}$ being decreasing. On the other hand, $X\leq_{\sIFR{2}}Y$ is equivalent to $c_2(x)=\TYs{2}^{-1}(\TXs{2}(x))$ being convex or, alternatively, $c_2^\prime$ being increasing. Differentiating, $c_2^\prime(x)=\frac{\TXs{1}(x)}{\TYs{2}(\TYs{1}^{-1}(\TXs{2}(x)))}=\frac{\TXs{1}(x)}{\TXs{2}(x)}$. Hence, $X\leq_{\sIFR{2}}Y$ is equivalent to $c_2^\prime(\TXs{1}^{-1}(x))=d(x)$ being decreasing, which proves the equivalence.
\end{proof}
As an immediate consequence, we have the hereditary property of the $\sIFR{s}$ order with respect to exponentially distributed random variables. This proves Nanda et al.~\cite{ASOK} remark for the particular choice of the exponential as the reference distribution.
\begin{cor}
Let $X$ be a random variable with distribution function $F_X\in\mathcal{F}$ and $Y$ a random variable with exponential distribution. If, for some $s\geq 1$, $X\leq_{\sIFR{s}}Y$, then $X\leq_{\sIFR{(s+1)}}Y$.
\end{cor}
However, the same hereditary does not hold when comparing general random variables with respect to the $\sIFR{s}$ ordering. That is, Remark~3.1 in Nanda et al~\cite{ASOK} is, in general, not true as it is proven in the proposition that follows.
\begin{prop}
\label{prop:counter}
Neither the $\sIFR{1}$ or the DMRL order imply the $\sIFR{2}$ order.
\end{prop}
\begin{proof}
Given $c_1,c_2>0$, we say that a random variable $X$ has branched Pareto distribution with parameters $c_1,c_2$,  $X\sim{\rm BP}(c_1,c_2)$, if its survival function is:
$$
\TXs{1}(x)=\frac{c_1^2}{(x+c_1)^2}\dI_{[0,c_1]}(x)+\frac{(c_1+c_2)^2}{4(x+c_2)^2}\dI_{(c_1,+\infty)}(x).
$$
Explicit expressions for the 2-iterated distribution and for the corresponding inverse functions are:
$$
\begin{array}{l}
\displaystyle\TXs{2}(x)=\frac{4}{3c_1+c_2}\left(\frac{c_1^2}{x+c_1}+\frac{c_2-c_1}{4}\right)\dI_{[0,c_1]}(x)+
\frac{(c_1+c_2)^2}{(3c_1+c_2)(x+c_2)}\dI_{(c_1,+\infty)}(x), \\ \\
\displaystyle\TXs{1}^{-1}(x)=\left(\frac{c_1+c_2}{2\sqrt{x}}-c_2\right)\dI_{[0,\frac14]}(x)+
    \left(\frac{c_1}{\sqrt{x}}-c_1\right)\dI_{(\frac14,+\infty)}(x), \\ \\
\displaystyle
\TXs{2}^{-1}(x)=\left(\frac{(c_1+c_2)^2}{(3c_1+c_2)x}-c_2\right)\dI_{[0,\frac{c_1+c_2}{3c_1+c_2}]}(x)
+\left(\frac{4c_1^2}{(3c_1+c_2)x-(c_2-c_1)}-c_1\right)\dI_{(\frac{c_1+c_2}{3c_1+c_2},+\infty)}(x).
\end{array}
$$
Moreover,
$$
\TXs{2}(\TXs{1}^{-1})(x)=\frac{2(c_1+c_2)\sqrt{x}}{3c_1+c_2}\dI_{[0,\frac{1}{4}]}(x)
+\frac{4}{3c_1+c_2}\left( c_1\sqrt{x}+\frac{c_2-c_1}{4}\right)\dI_{(\frac{1}{4},+\infty)}(x),
$$
and
$$
\TXs{1}(\TXs{2}^{-1})(x)=
\frac{(3c_1+c_2)^2}{4(c_1+c_2)^2} x^2\dI_{[0,\frac{c_1+c_2}{3c_1+c_2}]}(x)
+\frac{((3c_1+c_2)x-(c_2-c_1))^2}{16c_1^2}\dI_{(\frac{c_1+c_2}{3c_1+c_2},+\infty)}(x).
$$
Choosing suitably the parameters $c_1$ and $c_2$, we obtain the counter-example. A possible choice is $X\sim{\rm BP}(5,10)$ and $Y\sim{\rm BP}(2,6)$. For these parameters, we find
$$
d(x)=\frac{10}{9}\dI_{[0,\frac14)}(x)+\frac{25\left( 2\sqrt{x}+1\right)}{12\left( 5\sqrt{x}+\frac{5}{4}\right)}\dI_{[\frac14,+\infty)}(x),
$$
that is decreasing,
%
$$
c_2^\prime(x)=\frac{81}{100}\dI_{[0,\frac35]}(x)+\frac{9x^2}{(5x-1)^2}\dI_{(\frac35,\frac23]}(x)
+\frac{4(3x-1)^2}{(5x-1)^2}\dI_{(\frac23,1]}(x),
$$
which is not monotone, and
$$
c_1(x)=\TYs{1}^{-1}(\TXs{1})(x)=
\left(\frac{2(x+5)}{5}-2\right)\dI_{[0,5]}(x)+\left(\frac{8(x+10)}{15}-6\right)\dI_{(5,+\infty)}(x),
%
$$
which is convex.
%
%
%
\end{proof}

\section{A criterium for $\sIFR{s}$ ordering and a first application}
We have recalled (see Theorem~\ref{AOmain}), the criterium introduced by Arab and Oliveira~\cite{AO18} to prove $\sIFR{s}$ order between two different random variables, and we have mentioned, in Theorem~\ref{AOmain1}, the straightforward extension to prove the $\sIFRA{s}$ order. The criterium introduced in Theorem~\ref{AOmain} was used in Arab and Oliveira~\cite{AO18} to establish the iterated order within the families of the Gamma or the Weibull distributions. The proofs given in Arab and Oliveira~\cite{AO18} required a careful analysis of the sign variation of
$$
P_s(x)=\log f_Y(x)-\log f_X(ax+b)+\log\frac{\dE X^{s-1}}{a^{s}\dE Y^{s-1}},
$$
(or of $P_{s-1}$, defined in Theorem~\ref{AOmain}). A close look at those proofs shows that the difficult cases to handle always correspond to $b<0$, needing a correct positioning of the roots. So, it would be quite useful if we could reduce the need to verify the behaviour described in Theorem~\ref{AOmain}, considering only $a>0$ and $b\geq 0$. We may obtain such a simplification with the help of the $\sIFRA{s}$ ordering.
\begin{thm}
\label{thm:newcrit}
Let $X$ and $Y$ be random variables with distribution functions $F_X,F_Y\in\mathcal{F}$, respectively. If $X\leq_{\sIFRA{s}}Y$ and the criterium from Theorem~\ref{AOmain} is verified for $b\geq0$, then $X\leq_{\sIFR{s}}Y$.
\end{thm}
\begin{proof}
To prove that $X\leq_{\sIFR{s}}Y$, we need to verify that $c_s(x)=\TYs{s}^{-1}(\TXs{s}(x))$ is convex or, equivalently, that $\TXs{s}^{-1}(\TYs{s}(x))$ is concave. Taking into account Theorem~20 in Arab and Oliveira~\cite{AO18}, this is equivalent to verifying that $V(x)=\TXs{s}^{-1}(\TYs{s}(x))-(ax+b)$ has, for every real numbers $a$ and $b$, at most the sign variation ``$-,+,-$''. The assumption $X\leq_{\sIFRA{s}}Y$ means that $\frac{c_s(x)}{x}$ is increasing or, equivalently, that $\frac{\TXs{s}^{-1}(\TYs{s}(x))}{x}$ is decreasing. For $x>0$, the sign variation of $V(x)$ is the same as the sign variation of $\frac{V(x)}{x}=\left(\frac{\TXs{s}^{-1}(\TYs{s}(x))}{x}-a\right)-\frac{b}{x}$. The expression in the parenthesis is decreasing and, for $b<0$, $\frac{b}{x}$ is increasing, therefore $\frac{V(x)}{x}$ has, at most, one root, so the proof is concluded.
\end{proof}
We may now prove a comparison result ordering two distributions, one from the Weibull family and the other from the Gamma family.
\begin{prop}
\label{prop:WG1}
If $\alpha>1$, then ${\rm Weibull}(\alpha,\theta_1)\leq_{\sIFR{s}}\Gamma(\alpha,\theta_2)$, for every $s\geq 1$.
\end{prop}
\begin{proof}
Choose $X$ with ${\rm Weibull}(\alpha,1)$ distribution with density $f_X(x)=\alpha x^{\alpha-1}e^{-x^\alpha}$, and $Y$ with $\Gamma(\alpha,1)$ distributed with density $f_Y(x)=\frac{1}{\Gamma(\alpha)}x^{\alpha-1}e^{-x}$. We are taking $\theta_1=\theta_2=1$, as these are scale parameters, so their value does not affect the order relation between the random variables. We want to prove that $X\leq_{\sIFR{s}}Y$. Put $V_s(x)=\TYs{s}(x)-\TXs{s}(ax+b)$. We will now analyse the sign variation of $V_s$ for $x\geq 0$.
\begin{description}
\item[{\rm \textit{Step 1. The $\sIFRA{s}$ ordering}.}]
On the definition of $V_s$ take $b=0$. Therefore, we have
$$
P_s(x)=-(\alpha-1)\log(a)-x+a^\alpha x^\alpha-\log(\alpha\Gamma(\alpha))+\log\frac{\dE X^{s-1}}{a^s\dE Y^{s-1}},
$$
implying that $\lim_{x\rightarrow+\infty}P_s(x)=+\infty$, and $P_s^\prime(x)=-1+a^\alpha\alpha x^{\alpha-1}$, so the sign variation of $P_s^\prime$ is ``$-,+$'', and the monotonicity of $P_s$ is ``$\searrow\nearrow$''. If $P_s(0)<0$, the sign variation of $P_s$ is ``$-,+$'', so, using Lemma~\ref{sign-integral}, the sign variation of $V_s$ is, at most, ``$-,+$''. If $P_s(0)\geq 0$, the sign variation of $P_s$ may be ``$+,-,+$''. The function $V_s$ is obtained by integrating $H_s$, given in Theorem~\ref{AOmain}, so, again based on Lemma~\ref{sign-integral}, and taking into account that $V_s(0)=0$, the sign variation of $V_s$ is, at most, ``$-,+$''. Therefore, we have proved that $X\leq_{\sIFRA{s}}Y$.

\item[{\rm \textit{Step 2. The $\sIFR{s}$ ordering}.}]
We consider now $V_s$ with $a>0$ and $b>0$. Then we have
$$
P_s(x)=(\alpha-1)\left(\log x-\log(ax+b)\right)-x+(ax+b)^\alpha-\log(\alpha\Gamma(\alpha))+\log\frac{\dE X^{s-1}}{a^s\dE Y^{s-1}}.
$$
It is obvious that $\lim_{x\rightarrow+\infty}P_s(x)=+\infty$. Differentiating, we have that
$$
P_s^\prime(x)=\frac{\alpha-1}{x}-\frac{a(\alpha-1)}{ax+b}-1+a\alpha(ax+b)^{\alpha-1}
=\frac{N_s(x)}{x(ax+b)},
$$
where $N_s(x)=a\alpha x(ax+b)^\alpha-ax^2-bx+b(\alpha-1)$.
Hence, as we will be considering $x$ such that the denominator is positive, the sign of $P_s^\prime$ is determined by the sign of $N_s$. Differentiating $N_s$, we obtain
$
N_s^{\prime\prime\prime}(x)=a^3\alpha^2(\alpha-1)(ax+b)^{\alpha-3}\bigl(a(\alpha+1)x+3b\bigr).
$
Therefore, ${\rm sgn}(N_s^{\prime\prime\prime})={\rm sgn}\bigl(a(\alpha+1)x+3b\bigr)$.
As $a(\alpha+1)x+3b\geq 0$, it follows that $N_s^{\prime\prime\prime}(x)\geq 0$, hence $N_s^{\prime\prime}$ is increasing. We have that $\lim_{x\rightarrow+\infty}N_s^{\prime\prime}(x)=+\infty$, and $N_s^{\prime\prime}(0)=2a(a\alpha^2b^{\alpha-1}-1)$, and this last one may be either positive or negative. Looking now at $N_s^\prime$, we have $N_s^\prime(+\infty)=+\infty$, and $N_s^\prime(0)=b(a\alpha b^{\alpha-1}-1)$, which my be either positive or negative, irrespective to the sign at the origin for $N_s^{\prime\prime}$. Finally, we have $N_s(0)=b(\alpha-1)>0$ and $\lim_{x\rightarrow+\infty}N_s(x)=+\infty$. The table below summarizes the most sign varying possibilities, taking into account the behaviour just described.
\begin{center}
\begin{tabular}{lcccc}\hline
 & \multicolumn{2}{c}{$N^{\prime\prime}(0)>0$ } & \multicolumn{2}{c}{$N^{\prime\prime}(0)<0$ } \\
sign variation of $N^{\prime\prime}$ & \multicolumn{2}{c}{$+$} & \multicolumn{2}{c}{$-,+$} \\
monotonicity of $N^\prime$ & \multicolumn{2}{c}{$\nearrow$} & \multicolumn{2}{c}{$\searrow\nearrow$} \\ \hline
 & $N^{\prime}(0)>0$ & $N^{\prime}(0)<0$ & $N^{\prime}(0)>0$ & $N^{\prime}(0)<0$ \\
sign variation of $N^\prime$ & $+$ & $-,+$ & $+,-,+$ & $-,+$ \\
monotonicity of $N$ & $\nearrow$ & $\searrow\nearrow$ & $\nearrow\searrow\nearrow$ & $\searrow\nearrow$ \\ \hline
sign variation of $N$ & $+$ & $+,-,+$ & $+,-,+$ & $+,-,+$ \\ \hline
\end{tabular}
\end{center}
As ${\rm sgn}(P_s^\prime)={\rm sgn}(N_s)$, it follows that the possible monotonicities for $P_s$ are ``$\nearrow$'' or ``$\nearrow\searrow\nearrow$''. Going back to the expression for $P_s$, we verify that $\lim_{x\rightarrow0^+}P_s(x)=-\infty$ and $\lim_{x\rightarrow+\infty}P_s(x)=+\infty$, therefore, the possible sign variation of $P_s$ are ``$-,+$'' or ``$-,+,-,+$''.
Based again on Lemma~\ref{sign-integral}, for the first case it follows that the possible sign variations for $V_s$ are ``$-,+$'' or ``$+$'', while in the second case, the possible sign variation for $V_s$ are ``$-,+,-,+$'', ``$+,-,+$'', ``$-,+$'' or ``$+$''. Taking into account that $V_s(0)=1-\TXs{s}(b)\geq0$, actually only the sign variations starting at positive values are possible, that is, the possibilities are ``$+$'' or ``$+,-,+$'', so the proof is concluded.
\end{description}
\end{proof}
\begin{rem}
The proof of the comparison just described may be approached using Theorem~\ref{AOmain}, that is, the same methodology as in Arab and Oliveira~\cite{AO18}. In this case we would need to describe the sign variation also for the case $b<0$, and this can only be successfully completed assuming $\alpha>2$, due to the need to have a precise characterization of the location of the roots of $V_s$ enabling to derive the appropriate control of the sign variation of this function.
\end{rem}
Using the criterium proved in Theorem~\ref{thm:newcrit} we may complete the comparison within the Gamma or the Weibull families of distributions, partially given in Propositions~30--33 in Arab and Oliveira~\cite{AO18}. We state here the complete result. 
\begin{thm}
Let $\alpha^\prime>\alpha>0$.
\begin{enumerate}
\item
If $X\sim\Gamma(\alpha^\prime,\theta_1)$ and $Y\sim\Gamma(\alpha,\theta_2)$ then $X\leq_{\sIFR{s}}Y$.
\item
If $X\sim{\rm Weibull}(\alpha^\prime,\theta_1)$ and $Y\sim{\rm Weibull}(\alpha,\theta_2)$ then $X\leq_{\sIFR{s}}Y$.
\end{enumerate}
\end{thm}
\begin{proof}
Given Propositions~30--33 in Arab and Oliveira~\cite{AO18}, we only need to consider the case where $1>\alpha^\prime>\alpha>0$. The result follows repeating the steps for the proof of Proposition~\ref{prop:WG1}, with the arguments used in Propositions~30 and 32 in Arab and Oliveira~\cite{AO18}.
\end{proof}
\begin{rem}
The treatment of the case $1>\alpha^\prime>\alpha>0$ was out of reach of the methodology used in \cite{AO18}, exactly due to the difficulty on handling the sign variation of the function $V_s$ when choosing $b<0$.
\end{rem}

\section{Failure rate ordering of exponentially distributed parallel systems}
We now apply our results to prove extended ordering relations among parallel systems with components that have exponentially distributed lifetimes. We will be extending Theorem~3.1 by Kochar and Xu~\cite{KX09}, where these authors prove that a parallel system where the components have the same exponential distribution ages faster than a same sized system where the components have exponential lifetimes with different mean values.
%
We will be using the criterium introduced in Theorem~\ref{thm:newcrit} to extend, for general $s\geq 1$, this ageing characterization of parallel systems proved by Kochar and Xu~\cite{KX09}.

Throughout, this section we take
\begin{equation}
\label{maxdefs}
\begin{array}{lcp{8cm}}
X=\max(X_1,X_2), & \; & $X_1$ and $X_2$ are independent mean 1 exponentially distributed, \\
Y=\max(Y_1,Y_2), & & $Y_1$, mean 1 exponentially distributed, $Y_2$, mean $1/\lambda<1$ exponentially distributed, and independent.
\end{array}
\end{equation}
The choice made for the mean values of the components lifetimes is not really essential, but makes our proofs easier to explain. The only important fact is that $X_1$ and $X_2$ have the same mean. Indeed, taking into account the comments after Definition~\ref{DEF S-IFR}, we may always renormalize the variables to reduce to the present case. This section studies the $s-$iterated failure rate order between $X$ and $Y$. The main tool for the analysis is the result about roots of polynomials of exponentials, recalled in Theorem~\ref{thm:zeros}.

As already mentioned in course of the proof of Proposition~\ref{prop:ex1}, it is easily verified that
$$
\TYs{s}(x)=\frac{1}{c(s,\lambda)}\left(e^{-x}+\frac{e^{-\lambda x}}{\lambda^{s-1}}-\frac{e^{-(\lambda+1) x}}{(\lambda+1)^{s-1}}\right),
$$
where $c(s,\lambda)=1+\frac{1}{\lambda^{s-1}}-\frac{1}{(\lambda+1)^{s-1}}$. The tail of the distribution of $X$ is obtained replacing $\lambda$ by 1 in these expressions.

\smallskip

For sake of readability, we will present the various partial results leading to the comparison of $X$ and $Y$ in $\sIFR{s}$ order in a series of propositions.
\begin{prop}
\label{prop1}
Let $X$ and $Y$ be defined as in (\ref{maxdefs}). For every $s\geq 1$ and $x\geq 0$, we have $\TXs{s}(x)\geq\TYs{s}(x)$.
\end{prop}
\begin{proof}
Define
\begin{eqnarray*}
\lefteqn{U_s(x)=\TXs{s}(x)-\TYs{s}(x)} \\
 & & = \frac{2^s e^{-x}-e^{-2x}}{2^s-1}
-\frac{1}{c(s,\lambda)}\left(e^{-x}+\frac{e^{-\lambda x}}{\lambda^{s-1}}-\frac{e^{-(\lambda+1) x}}{(\lambda+1)^{s-1}}\right) \\
 & &=
\left(\frac{2^s}{2^s-1}-\frac{1}{c(s,\lambda)}\right)e^{-x}-\frac{e^{-2x}}{2^s-1}
  -\frac{e^{-\lambda x}}{c(s,\lambda)\lambda^{s-1}}
  +\frac{e^{-(\lambda+1)x}}{c(s,\lambda)(\lambda+1)^{s-1}}.
\end{eqnarray*}
We are considering $\lambda>1$, so the signs of the coefficients of $U_s$, after ordering decreasingly with respect to the exponents, are ``$+,-,-,+$'' (the sign of the coefficients of $e^{-\lambda x}$ and $e^{-2x}$ are the same, so we do not need to consider the two cases). So, taking into account Theorem~\ref{thm:zeros}, $U_s$ has, at most, two real roots. One root is easily located, as $U_s(0)=0$. Moreover, notice that $\lim_{x\rightarrow-\infty}U_s(x)=+\infty$, governed by the sign of the coefficient of $e^{-x}$, while $\lim_{x\rightarrow+\infty}U_s(x)=0^+$, described by the sign of the coefficient of the exponential with the smallest exponent.
In order to locate the remaining root, we need to differentiate: for $k<s$, we have
$$
U_s^{(k)}(x)=(-1)^k\left[
\frac{2^s e^{-x}-2^ke^{-2x}}{2^s-1}
-\frac{1}{c(s,\lambda)}\left(e^{-x}+\frac{e^{-\lambda x}}{\lambda^{s-1-k}}-\frac{e^{-(\lambda+1) x}}{(\lambda+1)^{s-1-k}}\right)
\right],
$$
and
$$
U_s^{(s)}(x)=(-1)^s\left[
\frac{2^s e^{-x}-2^se^{-2x}}{2^s-1}
-\frac{1}{c(s,\lambda)}\left(e^{-x}+\lambda e^{-\lambda x}-(\lambda+1)e^{-(\lambda+1) x}\right)
\right].
$$
Hence, the signs of the coefficients of the exponentials alternate with each differentiation, and $U_s^{(s)}(0)=0$. It is now convenient to separate into two cases.
\begin{description}
\item[$s$ even:]
The signs of the coefficients in $U_s^{(s)}$ are ``$+,-,-,+$'', implying that\linebreak $\lim_{x\rightarrow-\infty}U_s^{(s)}(x)=+\infty$, $\lim_{x\rightarrow+\infty}U_s^{(s)}(x)=0^+$, and $U_s^{(s)}$ has, at most, two real roots. As $U_s^{(s)}(0)=0$, depending on the location of the second root, the sign variation in $(0,+\infty)$ of $U_s^{(s)}$ may be either ``$+$'' or ``$-,+$''. The sign of $U_s^{(s-1)}(0)$ is not determined, and the signs of the limits at $\pm\infty$ are reversed with respect to $U_s^{(s)}$, so we need to consider the two possibilities leading to the following possible situations:
\begin{center}
\begin{tabular}{p{3.8cm}cccc}\hline
$U_s^{(s)}$ & \multicolumn{2}{c}{``$+$''} & \multicolumn{2}{c}{``$-,+$''} \\ \hline
$U_s^{(s-1)}(0)$ & positive & negative & positive & negative \\
sign variation of $U_s^{(s-1)}$ in $(0,+\infty)$ & not possible & ``$-$'' & ``$+,-$'' & ``$-$'' \\ \hline
\end{tabular}
\end{center}
Therefore, there are only two possible sign variations for $U_s^{(s-1)}$ when $x\in(0,+\infty)$: ``$-$'' or ``$+,-$''. We may proceed to characterize the possible sign variation in $(0,+\infty)$ of $U_s^{(s-2)}$, repeating the above arguments. Again, notice that it is not possible to determine the sign of $U_s^{(s-2)}(0)$, so the possibilities are:
\begin{center}
\begin{tabular}{p{3.8cm}cccc}\hline
$U_s^{(s-1)}$ & \multicolumn{2}{c}{``$-$''} & \multicolumn{2}{c}{``$+,-$''} \\ \hline
$U_s^{(s-2)}(0)$ & positive & negative & positive & negative \\
sign variation of $U_s^{(s-2)}$ in $(0,+\infty)$ & ``$+$'' & not possible & ``$+$'' & ``$-,+$'' \\ \hline
\end{tabular}
\end{center}
We find, for the sign variation in $(0,+\infty)$ of $U_s^{(s-2)}$ exactly the same behaviour as for $U_s^{(s)}$, so we may repeat the arguments above to find that $U_s^\prime$ has the same sign variation in $(0,+\infty)$ as $U_s^{(s-1)}$, that is, it is either ``$-$'' or ``$+,-$''. Going back to $U_s$ remember $\lim_{x\rightarrow-\infty}U_s(x)=+\infty$, $\lim_{x\rightarrow+\infty}U_s(x)=0^+$, and $U_s(0)=0$. This behaviour does not allow for the case $U_s^\prime$ being always negative, so the sign variation of $U_s^\prime$ is ``$+,-$'', which implies that $U_s(x)\geq 0$, for every $x\geq 0$.

\item[$s$ odd:]
For this case, we have that the signs of the coefficients in $U_s^{(s)}$ are ``$-,+,+,-$'', $\lim_{x\rightarrow-\infty}U_s^{(s)}(x)=-\infty$, $\lim_{x\rightarrow+\infty}U_s^{(s)}(x)=0^-$, $U_s^{(s)}(0)=0$, and $U_s^{(s)}$ has, at most, two real roots. Therefore, the only possible sign variation in $(0,+\infty)$ for $U_s^{(s)}$ is ``$+,-$'', which is what we found for the $s-1$ derivative in the previous case. Hence, repeating the arguments, we find the same conclusion, that is, $U_s(x)\geq 0$, for every $x\geq 0$, as well.
\end{description}
\end{proof}

\begin{cor}
Let $X$ and $Y$ be defined as in (\ref{maxdefs}). Then $\frac{\TYs{s}^{-1}\TXs{s}(x)}{x}\leq 1$, for every $x>0$.
\end{cor}
\begin{proof}
We have just proved that $\TXs{s}(x)\geq\TYs{s}(x)$ which, as $\TYs{s}$ is a decreasing function, implies that $\TYs{s}^{-1}\TXs{s}(x)\leq x$, so the result is proved.
\end{proof}

\begin{prop}
\label{prop2}
Let $X$ and $Y$ be defined as in (\ref{maxdefs}). For every $s\geq 1$, $X\leq_{\sIFRA{s}}Y$.
\end{prop}
\begin{proof}
We need to prove that $t_s(x)=\frac{\TYs{s}^{-1}(\TXs{s}(x))}{x}$ is increasing for $x\geq 0$, or, equivalently, that the sign variation in $(0,+\infty)$ of $t_s(x)-a$ is, at most, ``$-,+$''. The previous corollary means that we need only to consider $0<a\leq 1$. This is still equivalent to proving that, for the described choice for $a$, $\TXs{s}(x)-\TYs{s}(ax)$ behaves, at most, as ``$+,-$''. Reversing this expression, this is equivalent to prove that $V_s(x)=\TYs{s}(x)-\TXs{s}(ax)$ behaves, at most, as ``$-,+$'', now for $a>1$. This is the same formulation as in Theorem~\ref{convexity-equivalence} with a reduced scope for the choice of $a$. To write the expression explicitly, we have
$$
V_s(x)=
\frac{1}{c(s,\lambda)}\left(e^{-x}+\frac{e^{-\lambda x}}{\lambda^{s-1}}-\frac{e^{-(\lambda+1) x}}{(\lambda+1)^{s-1}}\right)
-\frac{2^s e^{-ax}-e^{-2ax}}{2^s-1}.
$$
We will be using Theorem~\ref{thm:zeros} to identify the maximum number the roots of $V_s$ and proceed in a similar way as before to locate them, and infer the sign variation of the function. As in the proof of Proposition~\ref{prop1}, we start by differentiating to obtain
$$
V_s^{(s)}(x)=(-1)^s\left[
\frac{1}{c(s,\lambda)}\left(e^{-x}+\lambda e^{-\lambda x}-(\lambda+1)e^{-(\lambda+1) x}\right)
-2^sa^s\frac{e^{-ax}-e^{-2ax}}{2^s-1}\right],
$$
so we have $V_s^{(s)}(0)=0$. To apply Theorem~\ref{thm:zeros}, we need to order decreasingly with respect to the exponents the exponential terms in $V_s$, which means we need to separate into several cases, depending on the location of $a$ with respect to $\lambda$, verifying in each one that the sign variation of $V_s$ is, at most, ``$-,+$''. Recall that both these parameters are larger or equal than 1.
\begin{description}
\item[Case 1: $1<a<2a<\lambda<\lambda+1$.]
As previously, we need to treat separately the case where $s$ is even from where $s$ is odd.
    \begin{description}
    \item[$s$ even:]
    The sign pattern of the coefficients in $V_s$ and in $V_s^{(s)}$ is now ``$+,-,+,+,-$'', indicating that each function has, at most, three real roots. Moreover, this sign pattern implies that $\lim_{x\rightarrow-\infty}V_s^{(s)}(x)=-\infty$, $\lim_{x\rightarrow+\infty}V_s^{(s)}(x)=0^+$, and $V_s^{(s)}(0)=0$. Therefore, the most sign  varying in $(0,+\infty)$ case for $V_s^{(s)}$ is ``$+,-,+$''. As before, the sign of $V_s^{(s-1)}(0)$ is not determined, so we need to analyze each possibility. Remember that the coefficient signs and the signs at each limit when $x\rightarrow\pm\infty$ reverses with each differentiation, meaning that $\lim_{x\rightarrow-\infty}V_s^{(s-1)}(x)=+\infty$, $\lim_{x\rightarrow+\infty}V_s^{(s-1)}(x)=0^-$. Taking this into account, the possibilities are:
    \begin{center}
    \begin{tabular}{p{3.8cm}cccc}\hline
    $V_s^{(s)}$ & \multicolumn{4}{c}{``$+,-,+$''}  \\ \hline
    $V_s^{(s-1)}(0)$ & \multicolumn{2}{c}{positive} & \multicolumn{2}{c}{negative} \\
    sign variation of $V_s^{(s-1)}$ in $(0,+\infty)$ & \multicolumn{2}{c}{``$+,-$''} & \multicolumn{2}{c}{``$-,+,-$''} \\ \hline
    \end{tabular}
    \end{center}
    To proceed the analysis about $V_s^{(s-2)}$, notice first that the sign of this function at the origin is not determined, and that $\lim_{x\rightarrow-\infty}V_s^{(s-2)}(x)=-\infty$, $\lim_{x\rightarrow+\infty}V_s^{(s-2)}(x)=0^+$. Therefore, the possible sign variations are:
    \begin{center}
    \begin{tabular}{p{3.8cm}cccc}\hline
    $V_s^{(s)}$ & \multicolumn{4}{c}{``$+,-,+$''}  \\ \hline
    $V_s^{(s-1)}(0)$ & \multicolumn{2}{c}{positive} & \multicolumn{2}{c}{negative} \\
    sign variation of $V_s^{(s-1)}$ in $(0,+\infty)$ & \multicolumn{2}{c}{``$+,-$''} & \multicolumn{2}{c}{``$-,+,-$''} \\ \hline
    $V_s^{(s-2)}(0)$ & positive & negative & positive & negative \\
    sign variation of $V_s^{(s-2)}$ in $(0,+\infty)$ & ``$+$'' & ``$-,+$'' & ``$+,-,+$'' & ``$-,+$'' \\ \hline
    \end{tabular}
    \end{center}
    This means that the most sign varying possibility for $V_s^{(s-2)}$ is the same as for $V_s^{(s)}$, hence we may recurse on the argument to arrive at the conclusion that the most sign varying in $(0,+\infty)$ case for $V_s^\prime$ is ``$-,+,-$''. Taking into account that $V_s(0)=0$, $\lim_{x\rightarrow-\infty}V_s(x)=-\infty$, $\lim_{x\rightarrow+\infty}V_s(x)=0^+$, and $V_s$ has, at most, three real roots, its sign variation in $(0,+\infty)$ may, at most, be ``$-,+$''.
    \item[$s$ odd:]
    The sign pattern of the coefficients in $V_s^{(s)}$ is now ``$-,+,-,-,+$'', and $\lim_{x\rightarrow-\infty}V_s^{(s)}(x)=+\infty$, $\lim_{x\rightarrow+\infty}V_s^{(s)}(x)=0^-$, and $V_s^{(s)}(0)=0$. This means that the most sign varying in $(0,+\infty)$ possibility for $V_s^{(s)}$ is now ``$-,+,-$'', which corresponds to the behaviour of the $s-1$ derivative in the previous case, so the same conclusion still holds.
    \end{description}
    Therefore, for this case we have verified that the sign variation in $(0,+\infty)$ of $V_s$ is, at most, ``$-,+$''.

\item[Case 2: $1<a<\lambda<2a<\lambda+1$.]
The sign pattern of the coefficients coincides with the one observed in the previous case, so the result also holds.

\item[Case 3: $1<a<\lambda<\lambda+1<2a$.]
The sign pattern of the coefficients of the function $V_s$ is now ``$+,-,+,-,+$'', hence there could exist up to 4 real roots. Of course, we still have $V_s(0)=0$, so we need to locate the remaining ones. Due to the number of possible roots, a direct usage of the arguments as in the previous cases with $V_s$ does not allow to conclude about a sign variation compatible with $\sIFRA{s}$ order. Note that, in this case, we have $a>\frac{\lambda+1}{2}$, so, for every fixed $x\geq 0$, $\TXs{s}(ax)<\TXs{s}(\frac{\lambda+1}{2}x)$, therefore $V_s(x)=\TYs{s}(x)-\TXs{s}(ax)>V_{\ast,s}(x)=\TYs{s}(x)-\TXs{s}(\frac{\lambda+1}{2}x)$. We shall prove that
$V_{\ast,s}(x)\geq 0$, for $x\geq0$, so the same holds for $V_s$.
Rewriting $V_{\ast,s}$, with the exponentials already ordered decreasingly with respect to their exponents, we have
$$
V_{\ast,s}(x)=\frac{1}{c(s,\lambda)}e^{-x}
  -\frac{2^s}{2^s-1}e^{-\frac{(\lambda+1)}{2}x}
  +\frac{1}{c(s,\lambda)\lambda^{s-1}}e^{-\lambda x}
  +\left(\frac{1}{2^s-1}-\frac{1}{c(s,\lambda)(\lambda+1)^{s-1}}\right)e^{-(\lambda+1)x}.
$$
The coefficient of the last exponential is easily seen to be positive, so the sign pattern of the coefficients in $V_{\ast,s}$ is ``$+,-,+,+$'', hence, besides having $V_{\ast,s}(0)=0$, we have $\lim_{x\rightarrow-\infty}V_{\ast,s}(x)=+\infty$ and $\lim_{x\rightarrow+\infty}V_{\ast,s}(x)=0^+$. Moreover, taking into account Theorem~\ref{thm:zeros}, $V_{\ast,s}$ has, at most, two real roots. To complete the study of the sign variation we need, as before, to separate the cases depending on the value of $s$.
    \begin{description}
    \item[$s$ even.]
    Repeating the arguments above, the sign pattern for the coefficients of $V_{\ast,s}^{(s)}$ is the same as for $V_{\ast,s}$. Therefore, we may repeat the arguments used in course of proof of Proposition~\ref{prop1} for the case where $s$ is even, to derive that $V_{\ast,s}(x)\geq 0$, hence $V_s(x)\geq 0$, for every $x\geq 0$.
    \item[$s$ odd.]
    As in the proof of Proposition~\ref{prop1}, this corresponds to the behaviour of the $s-1$ derivative when $s$ is even, so the result also holds.
    \end{description}

\item[Case 4: $1<\lambda<a<2a<\lambda+1$.]
The sign pattern of the coefficients, after ordering the exponentials, is ``$+,+,-,+,-$'', meaning that are, at most, three real roots. This is exactly the same sign pattern we found in \textbf{Case 1} above. So, repeating the arguments, the same sign variation for $V_s$ follows.

\item[Case 5: $1<\lambda<a<\lambda+1<2a$.]
This is the simplest case to analyse. The sign pattern for the coefficients of $V_s$ is ``$+,+,-,-,+$'', implying that there are, at most, two real roots, $\lim_{x\rightarrow-\infty}V_s(x)=+\infty$ and $\lim_{x\rightarrow+\infty}V_s(x)=0^+$. This is easily seen to be compatible with two possible sign variation in $(0,+\infty)$: ``$-,+$'' or ``$+$''.

\item[Case 6: $1<\lambda<\lambda+1<a<2a$.]
This case produces the same sign pattern for the coefficients as for \textbf{Case 5}, so the same conclusion about the sign variation of $V_s$ follows.
\end{description}
Therefore, we have verified that in all possible cases, the sign variation of $V_s$ is, at most, ``$-,+$'', hence $t_s(x)$ is, for $x\geq 0$, increasing, so the proposition is proved.
\end{proof}
The previous result establishes the $\sIFRA{s}$ order, so may now proceed to the proof of the $\sIFR{s}$ relation between these two random variables.
\begin{thm}
\label{prop3}
Let $X$ and $Y$ be defined as in (\ref{maxdefs}). For every $s\geq 1$, $X\leq_{\sIFR{s}}Y$.
\end{thm}
\begin{proof}
The plan for the proof is the same as for Proposition~\ref{prop2}. The difference here is that we will be interested in proving the convexity of the relevant functions and we will not be able to automatically locate one of their roots. Taking into account Theorem~\ref{convexity-equivalence}, Remark~\ref{AOalpha} and Theorem~\ref{thm:newcrit} it is sufficient to verify that $V_s(x)=\TYs{s}(x)-\TXs{s}(ax+b)$ changes sign at most twice, in the order ``$+,-,+$'', for every $a>0$ and $b\geq 0$. The case $b=0$ was treated in Proposition~\ref{prop2}, so we may assume in the sequel that $b>0$. Although the function is similar to the one considered on Proposition~\ref{prop2}, one should notice that now $V_s(0)=1-\TXs{s}(b)>0$. Of course, the sign patterns of the coefficients are similar, but we must take into account the extra terms $e^{-b}$ and $e^{-2b}$.

\smallskip

We start by writing explicitly the expression for $V_s$ and its $s$-order derivative $V_s^{(s)}$:
$$
\begin{array}{l}
\displaystyle V_s(x)=
\frac{1}{c(s,\lambda)}\left(e^{-x}+\frac{e^{-\lambda x}}{\lambda^{s-1}}-\frac{e^{-(\lambda+1) x}}{(\lambda+1)^{s-1}}\right)
-\frac{2^s e^{-(ax+b)}-e^{-2(ax+b)}}{2^s-1}, \\ \\
\displaystyle V_s^{(s)}(x)=(-1)^s\left[
\frac{1}{c(s,\lambda)}\left(e^{-x}+\lambda e^{-\lambda x}-(\lambda+1)e^{-(\lambda+1) x}\right)
-2^sa^s\frac{e^{-(ax+b)}-e^{-2(ax+b)}}{2^s-1}\right].
\end{array}
$$
Note that $V_s^{(s)}(0)=\frac{(-1)^{s+1}2^s a^se^{-b}(1-e^{-b})}{2^s-1}$, which has the same sign as $(-1)^{s+1}$, as $b>0$.
\begin{description}
\item[Case 1: $1<a<2a<\lambda<\lambda+1$.]
The sign pattern of the coefficients, after ordering the exponential in decreasing order of their exponents, is ``$+,-,+,+,-$'', so $V_s$ has, at most, three real roots. Moreover, $\lim_{x\rightarrow-\infty}V_s(x)=-\infty$, $\lim_{x\rightarrow+\infty}V_s(x)=0^+$ so, remembering that $V_s(0)>0$, it follows that the sign variation of $V_s$ in $(0,+\infty)$ is either ``$+$'' or ``$+,-,+$''.

\item[Case 2: $1<a<\lambda<2a<\lambda+1$.]
This case is treated exactly as the previous one.

\item[Case 3: $1<a<\lambda<\lambda+1<2a$.]
As before, this case requires a more careful analysis, as the number of possible roots is larger. The sign pattern of the coefficients of the function $V_s$ is now ``$+,-,+,-,+$'', so we may have up to 4 real roots for $V_s$, and $\lim_{x\rightarrow-\infty}V_s(x)=+\infty$, $\lim_{x\rightarrow+\infty}V_s(x)=0^+$. We again separate according to s being even or odd.
\begin{description}
    \item[$s$ even.]
    In this case, we have $V_s^{(s)}(0)<0$, so the sign variation in $(0,+\infty)$ of $V_s^{(s)}$ is either ``$-,+,-,+$'' or `$-,+$''. Now, taking into account that each differentiation step reverses all signs, and that, with the exception of the $s$-order derivative, the sign of the derivatives at the origin is not determined, we have the following possibilities for the sign variations:
    \begin{center}
    {\footnotesize
    \begin{tabular}{p{2.5cm}cccc}\hline
    $V_s^{(s)}$ & \multicolumn{4}{c}{``$-,+$'' or ``$-,+,-,+$''}  \\ \hline
    $V_s^{(s-1)}(0)$ & \multicolumn{2}{c}{positive} & \multicolumn{2}{c}{negative} \\
    sign variation of $V_s^{(s-1)}$ in $(0,+\infty)$ & \multicolumn{2}{c}{``$+,-,+,-$'' or ``$+,-$''} & \multicolumn{2}{c}{``$-$'' or ``$-,+,-$''} \\ \hline
    $V_s^{(s-2)}(0)$ & positive & negative & positive & negative \\
    sign variation of $V_s^{(s-2)}$ in $(0,+\infty)$ &
    $\begin{array}{c}\mbox{``}+\mbox{''}\\\mbox{or}\\\mbox{``}+,-,+\mbox{''}\end{array}$ &
    $\begin{array}{c}\mbox{``}-,+\mbox{''}\\\mbox{or}\\\mbox{``}-,+,-,+\mbox{''}\end{array}$
    &
    $\begin{array}{c}\mbox{``}+\mbox{''}\\\mbox{or}\\\mbox{``}+,-,+\mbox{''}\end{array}$
    & ``$-,+$'' \\ \hline
    \end{tabular}}
    \end{center}
    Hence, the most sign varying possibility, in $(0,+\infty)$, for $V_s^{(s-2)}$ is the same as for $V_s^{(s)}$, so we repeat the argument to obtain that the most sign varying possibility, in $(0,+\infty)$, for $V_s^\prime$ is ``$+,-,+,-$''. Therefore, the monotonicity of $V_s$ is ``$\nearrow\searrow\nearrow\searrow$'', which, remembering that $V_s(0)>0$ implies that the sign variation of $V_s$ may be ``$+$'' or ``$+,-,+$''.
    \item[$s$ odd.]
    Now we have $V_s^{(s)}(0)>0$ which, taking into account the signs for $V_s^{(s)}$ at $\pm\infty$, implies a sign variation, in $(0,+\infty)$, as ``$+,-$'' or ``$+,-,+,-$'', that is, we find the same behaviour as for the $s-1$ derivative in the previous case, so the conclusion about the sign variation of $V_s$ also follows.
    \end{description}

\item[Case 4: $1<\lambda<a<2a<\lambda+1$.]
This case coincides with the behaviour observed for \textbf{Case 1} above, so the conclusion holds.

\item[Case 5: $1<\lambda<a<\lambda+1<2a$.]
The sign pattern of the coefficients of $V_s$, after ordering the exponentials in the usual way, is ``$+,+,-,-,+$'', implying that there are, at most, two real roots, $\lim_{x\rightarrow-\infty}V_s(x)=+\infty$ and $\lim_{x\rightarrow+\infty}V_s(x)=0^+$. As $V_s(0)>0$, the only possibility for the sign variation in $(0,+\infty)$ for $V_s$ is ``$+$'' or ``$+,-,+$''.

\item[Case 6: $1<\lambda<\lambda+1<a<2a$.]
This case coincides with the previous one.

\item[Case 7: $0<a<1$.]
In this case, regardless of the actual value for $a$, the sign pattern of the coefficients is ``$-,+,+,+,-$'', so $\lim_{x\rightarrow-\infty}V_s(x)=-\infty$ and $\lim_{x\rightarrow+\infty}V_s(x)=0^-$. As $V_s(0)>0$, the only possible sign variations in $(0,+\infty)$ is ``$+,-$''.
\end{description}
We have verified that, for all relevant choices of the parameters $a$ and $b$, the assumptions of Theorem~\ref{convexity-equivalence} are satisfied, so it follows that $X\leq_{\sIFR{s}}Y$.
\end{proof}
\begin{rem}
Our Theorem~\ref{prop3} above partially extends Theorem~3.1 in Kochar and Xu~\cite{KX09} to iterated failure ordering. The extension is partial as Theorem~\ref{prop3} deals only with maxima between two random variables, while Kochar and Xu's result deals with arbitrary families of variables.
\end{rem}

 In the result that follows, we compare the ageing properties of a parallel system with $n$ components with independent and identically distributed exponential lifetimes, with a parallel system with $k$ components that are independent and identical exponential lifetimes but assuming that  $k<n$.

\begin{prop}
%
Let $X_1,\ldots,X_m$, $m\geq 3$, be independent random variables with exponential distribution with mean $1/\lambda$, and $Y_1,\ldots,Y_k$, $2\leq k<m$, independent exponential random variables with mean $1/\beta$. If  $X_{(m)}=\max(X_1,\ldots,X_m)$ and $Y_{(k)}=\max(Y_1,\ldots,Y_k)$, then $X_{(m)}\leq_{\sIFR{1}}Y_{(k)}$
\end{prop}
\begin{proof}
As the parameters $\lambda$ and $\beta$ are scale parameters we may take $\lambda=\beta=1$. The random variables $X_m$ and $Y_k$ have the following tail distributions: $\overline{F}_m(x)=1-(1-e^{-x})^m$, $\overline{F}_k(x)=1-(1-e^{-x})^k$, respectively. The proposition follows by proving that $c_1(x)=\overline{F}^{-1}_k(\overline{F}_m)=-\log(1-(1-e^{-x})^{\frac{m}{k}})$ is convex. As $c_1(x)$ is differentiable, the convexity is characterized by the nonnegativeness of the second derivative. Computing derivatives, and taking into account that the sign of $c_1^{\prime\prime}$ is determined by the sign of its numerator, it can be seen that the sign of $c_1^{\prime\prime}$ is the same as the sign of
$$
Q(x)=\frac{m}{k}e^{-x}+(1-e^{-x})^{\frac{m}{k}}-1,
$$
which is positive for every $x\geq 0$ and $2\leq k<m$, so the conclusion follows.
\end{proof}
As an immediate consequence, we have the following ordering for the order statistics.
\begin{cor}
Let $X_1,\ldots,X_n$, $n\geq 3$, be independent random variables with exponential distribution with mean $1/\lambda$, and define, for each $k=2,\ldots,n$,\linebreak $X_{(k)}=\max(X_1,\ldots,X_k)$. Then $X_{(n)}\leq_{\sIFR{1}}X_{(n-1)}\leq_{\sIFR{1}}\cdots\leq_{\sIFR{1}}X_{(2)}$.
\end{cor}

\begin{rem}
Kochar and Xu~\cite{KX09} mention an unsolved problem for which they announce having empirical evidence although no mathematical proof could be obtained. This unsolved problem is stated as follows: let $X_i$ be independent exponentially distributed variables with means $1/\lambda_i$, and $Y_i$ be independent exponentially distributed variables with means $1/\theta_i$; if $(\lambda_1,\ldots,\lambda_n)\prec(\theta_1,\ldots,\theta_n)$, in the sense of Definition~A.1 in Marshal and Olkin~\cite{MO79} i.e.  $\sum_{i=1}^k\lambda_{(i)}\geq\sum_{i=1}^k\theta_{(i)}$, for $k=1,\ldots,n-1$, and $\sum_{i=1}^n\lambda_i=\sum_{i=1}^n\theta_i$ where $\lambda_{(1)}\leq \lambda_{(2)}\leq\cdots\leq\lambda_{(n)}$ then one should expect\linebreak $\max(X_1,\ldots,X_n)\leq_{\sIFR{1}}\max(Y_1,\ldots,Y_n)$. For the particular conjecture, we have evidence that it is in general not true for iterated failure rate order as long as the iteration parameter $s\geq 2$. We provide an example were this ordering does not hold. Let $X_1$, and $X_2$ be independent exponential random variables with means $1/\lambda_i$, $i=1,2$, and $Y_1$ and $Y_2$ be independent exponential random variables with means $1/\theta_i$, $i=1,2$. Assume, without loss of generality, that $\lambda_1\leq \lambda_2$ and $\theta_1\leq\theta_2$, and that $(\lambda_1,\lambda_2)\prec(\theta_1,\theta_2)$, i.e. $\lambda_1+\lambda_2 = \theta_1+\theta_2$ and $\lambda_1\geq \theta_1$. Write, for simplicity, $X=\max\{X_1,X_2\}$ and $Y=\max\{Y_1,Y_2\}$ and consider $V_s(x)=\TYs{s}(x)-\TXs{s}(ax)$, where $a>0$. If we choose the parameters $(s,\frac{\lambda_1}{\lambda_2},\frac{\theta_2}{\theta_1}, a) = (2,0.34,11,2.89)$ the sign variation of $V_s(x)$ is ``$-,+,-$'' so, according to Theorem~\ref{convexity-equivalence}, $X$ and $Y$ are not comparable with respect to the $\sIFRA{2}$ order, hence they cannot be comparable with respect to the $\sIFR{2}$ order. The particular choice for the parameters made above gives raise to a family of possible choices for the vectors $(\lambda_1,\lambda_2)$ and $(\theta_1,\theta_2)$ leading to counter-examples. Indeed, taking into account the order relation $(\lambda_1,\lambda_2)\prec(\theta_1,\theta_2)$, it follows that $(\lambda_1,\lambda_2,\theta_1,\theta_2)=\frac{1}{1474}(408,1200,134,1474)\vartheta$, with $\vartheta>0$, generates a whole family of counter-examples for the conjecture when the iteration parameter $s=2$. Alike Kochar and Xu, we cannot find a counter-example for the case $s=1$, nor provide a proof for such a result.
%
%
\end{rem}


%
%
%
%

\end{document}